\documentclass[12pt]{article}
\usepackage{graphicx}
\usepackage[all,2cell]{xy}          
\usepackage{amsthm}
\usepackage{tikz-cd}
\usepackage{mathtools}
\usepackage{mathrsfs}
\usepackage{enumerate}
\usepackage[colorlinks=true]{hyperref}
\hypersetup{allcolors=[rgb]{0.1,0.1,0.4}}

\usepackage{txfonts}

\usepackage{makeidx}
\makeindex

\usepackage[all,2cell]{xy}
\UseAllTwocells
\xyoption{arc}
\xyoption{rotate}

\usepackage{hyperref}   
\hypersetup{colorlinks, 
breaklinks,             
linkcolor=black,        
urlcolor=black}         

\newcommand{\fcat}[1]{\mathbf{#1}}      

\newcommand{\clm}{\mbox{Colim}}

\DeclareMathOperator{\Lan}{Lan}         

\newcommand{\op}{\mathrm{op}}           
        
\newcommand{\Set}{\fcat{Set}}           

\newcommand{\ladj}{\dashv}              

\newcommand{\oppair}[4]{%
\xymatrix{%
#1 \ar@<.5ex>[r]^-{#3} &{#2}\ar@<.5ex>[l]^-{#4}%
}}

\newcommand{\oppairi}[4]{%
\xymatrix@1{%
#1 \ar@<.5ex>[r]^{#3} &{#2}\ar@<.5ex>[l]^{#4}%
}}

\newcommand{\parpair}[4]{%
\xymatrix{%
#1 \ar@<.5ex>[r]^{#3} \ar@<-.5ex>[r]_{#4} &#2%
}}

\newcommand{\parpairi}[4]{%
\xymatrix@1{%
#1 \ar@<.5ex>[r]^{#3} \ar@<-.5ex>[r]_{#4} &#2%
}}

\newcommand{\adjn}[4]{%
\xymatrix{
#1 \ar@{}[d]|\ladj \ar@<1ex>[d]^{#4} \\
#2 \ar@<1ex>[u]^{#3}
}}

\newcommand{\hadjnli}[4]{%
\xymatrix@1{
#1 \ar@<1.1ex>[r]^-{#3} \ar@{}[r]|-\bot &#2 \ar@<1.1ex>[l]^-{#4}}}

\newcommand{\hadjnri}[4]{%
\xymatrix@1{
#1 \ar@<1.1ex>[r]^-{#3} \ar@{}[r]|-\top &#2 \ar@<1.1ex>[l]^-{#4}}}

\newtheorem{exem}{Exemple}

\newtheorem{coro}{Corollaire}
\newtheorem{thm}{Th\'eor\`eme}

\newtheorem{prop}{Proposition}
\newtheorem{defn}{D\'efinition}
\newtheorem{exer}{Exercice}

\newcommand{\ra}{\longrightarrow}

\renewcommand{\1}{{\bf 1}}

\newcommand{\n}{\mathbb{N}}
\newcommand{\R}{\mathbb{R}}
\newcommand{\Z}{\mathbb{Z}}
\newcommand{\Q}{\mathbb{Q}}

\newcommand{\ds}{\displaystyle}

\newcommand{\SET}{\mbox{\bf SET}}
\newcommand{\lan}{\mbox{Lan}}
\newcommand{\Ran}{\mbox{Ran}}
\def\subclassname{{\bfseries Mathematics Subject Classification
(2000)}\enspace}
\def\subclass#1{\par\addvspace\medskipamount{\rightskip=0pt plus1cm
\def\and{\ifhmode\unskip\nobreak\fi\ $\cdot$
}\noindent\subclassname\ignorespaces#1\par}}
\usepackage[utf8x]{inputenc}
\newcommand{\bv}{\begin{verbatim}}
\newcommand{\env}{\end{verbatim}}
\begin{document}

\title{Optimisation Abstraite\footnote{Expos\'e aux s\'eminaires de laboratoire LATAO, Facult\'e des Sciences de Tunis, Jeudi 18-2-2021.  }}


    \author{Fethi Kadhi }


\date{ }

\maketitle
\begin{abstract}
Les concepts de base en théorie des catégories sont les représentables, les adjoints, les limites et les monades.
Dans cet expos\'e, on définit la notion d'extension de Kan et on montre que cette notion couvre ces concepts. 
\end{abstract}


\section{Introduction}\label{s1}
Depuis sa naissance, la théorie des catégories représente un outil d'unification de différentes structures en mathématiques.
Les ensembles, les monoïdes, les groupes, les anneaux, les espaces vectoriels, les algèbres sont des catégories.
Dans un autre niveau d'abstraction, ces structures sont des objets d'une même catégorie.
L'existence d'objets et de flèches spéciaux dans une catégorie \`a donn\'e lieu aux concepts catégoriques de base,
 tels que les représentables, les limites, les adjoints et les monades. Fidèle \`a son rôle d'unification, la théorie des catégorie
fournit un concept qui permet d'unifier ces concepts de base. En computer Science, les extensions de Kan sont appliquees pour optimiser
les programmes fonctionnels \cite{HR}.

\section{Définitions et exemples}
\begin{defn}
Soient $A\stackrel{X}{\longrightarrow}C$ et $A\stackrel{K}{\longrightarrow}B$ deux foncteurs de même domaine.
Une extension \`a gauche de $X$ le long de $K$ consiste en une  paire $(L,\eta)$ o\`u $L:B\longrightarrow C$ est un foncteur et
$\eta:X\Longrightarrow LK$ est une transformation naturelle
$$	
\xymatrix{
A \ar[rr]^{X} \ar[dr]_{K}& & C\\
    & B \ar@{<=}[u]_{\eta}\ar[ur]_L 
} 
$$
universelle dans le sens suivant:
Pour un foncteur $B\stackrel{L'}{\longrightarrow}C$ et une  transformation naturelle $X\stackrel{\eta^{'}}{\Longrightarrow}L'K$,
$$
\xymatrix{
A \ar[rr]^{X} \ar[dr]_{K}& & C\\
    & B \ar@{<=}[u]_{\eta'}\ar[ur]_{L^{'}} 
} 
$$
il existe une unique transformation naturelle  $L\stackrel{\alpha}{\Longrightarrow}L'$\\
\begin{center}
\begin{tikzcd}[row sep=1cm, column sep=1cm]
A  \ar[dr, "K"', ""{name=K,near end}]
            \ar[rr, "X", ""{name=X, below,bend right}]&& C\\
& B    \ar[ur, bend left, "L", ""{name=L, below}]
                \ar[ur, bend right, "L'"', ""{name=LL}]
%
\arrow[Rightarrow, "\exists !\alpha", from=L, to=LL]
\arrow[Rightarrow, from=X, "\eta",below]
\end{tikzcd}
\end{center}
 telle que $\eta'=\alpha K.\eta$. C'est \`a dire $\eta'$ se factorise de façon unique suivant $\eta$.
$$
\xymatrix{
X \ar@{=>}[rr]^{\eta} \ar@{=>}[dr]_{\eta'}& & LK\ar@{=>}[ld]^{\alpha K}\\
    & L'K 
} 
$$
On \'ecrit: $\lan_K(X)=L$
\end{defn}
On voit que $\lan_K(X)$ parait comme un objet initial dans la catégorie de quelques foncteurs $L'\in C^B$.
C'est pour cette raison qu'on dit que $\lan_K(X)$, s'il existe, représente la façon optimale pour étendre le fonceur
$X$ du domaine de $X$ au codomaine de $K$.
\begin{exem} Un représentable est une extension de Kan:\\
Soit $1$ la catégorie terminale formée par un seul objet $1$ reli\'e \`a lui même par l'unique flèche $1_1$.
Soit $*:1\longrightarrow\Set$ le foncteur qui associe \`a l'objet $1$ un singleton $\{*\}$.
Soient $C$ une catégorie localement petite et $c$ un objet de $C$. On peut regarder $c$ comme un foncteur
$c:1\longrightarrow C$ avec $c(1)=c$ et $c(1_1)=1_c$. Est ce qu'il existe une façon optimale pour étendre
$*$ de $1$\`a $C$?
\end{exem}
 En optimisation, généralement, on procède de deux manières: ou bien on applique une série de conditions nécessaires d'optimalit\'e a fin de localiser le candidat optimal, ou bien on choisit un bon candidat et on montre que la fonction coût est optimale en ce candidat. Dans cet exemple on va choisir le bon candidat et vérifier qu'il est optimal. Dans un prochain paragraphe, nous allons présenter une formule qui permet d'exprimer ponctuellement le candidat optimal.
Soit $H^c:C\longrightarrow\Set$ définie par $H^c(x)=C(c,x)$. Montrons que $\lan_c(*)=H^c$.
\begin{proof}
Une transformation naturelle $\eta:*\Longrightarrow H^cc$ est définie par sa composante suivant $1$, $\eta_1:1\longrightarrow C(c,c)$.
$\eta_1$ consiste donc en un choix d'un élément de l'ensemble $C(c,c)$.
 Soit  $\eta_1=1_c$. On a besoin de vérifier que la paire $(H^c,\eta)$ est universelle dans le sens qui permet d'\'ecrire $\lan_c(*)=H^c$.
Soient $F: C\longrightarrow\Set$ et $\gamma:*\Longrightarrow Fc$. $\gamma_1:1\longrightarrow FC$ consiste en un élément $x\in Fc$.
D'après la démonstration du Lemme de Yoneda, il existe une unique transformation naturelle $\alpha:H^c\Longrightarrow F$ qui vérifie
$\alpha_c(1_c)=x$, c'est \`a dire $\alpha c\circ\eta=\gamma$. Ainsi $\lan_c(*)=H^c$.
\end{proof}
Ce premier exemple nous permet de regarder un représentable comme une extension de Kan.
Dans des prochains exemples, nous allons voir que les limites, les adjoints et les monades sont aussi des extensions de Kan.
C'est dans ce sens que Maclane \cite{ML} disait que tous les concepts sont des extensions de Kan.
En fait, la théorie des catégories a été crée pour unifier des constructions algébriques et géométriques différentes.
L'extension de Kan fait la même chose dedans la catégories elle-même. \\ 
La propriété universelle qui caractérise $(\lan_K(X),\eta)$ peut être exprimée par l'isomorphisme naturel en $H\in C^B$
\begin{equation}
C^B(\lan_K(X),H)\cong C^A(X,HK)
\label{eq:1}
\end{equation}
$\eta:X\Longrightarrow\lan_K(X)K$ est la transformation naturelle qui correspond \`a $1_{\lan_K(X)}$.\\
En passant \`a un univers ensembliste plus grand, on peut interpréter $\lan_K(X)$ de deux manières:
\begin{enumerate}
	\item $\lan_K(X)$ est une représentation du foncteur:
	$$ \begin{array}{rcl}
C^A(X,-.K):C^B&\ra & \SET\\
     H&\longmapsto&C^A(X,HK)										
\end{array}$$
	\item Si $\lan_K(X)$ existe pour tout $X\in C^A$ alors on peut regarder 
	$$ \begin{array}{rcl}
\lan_K(.):&C^A\longrightarrow& C^B\\
     X&\longmapsto&\lan_K(X)										
\end{array}$$
	
	comme l'adjoint \`a gauche du foncteur
	$$ \begin{array}{rcl}
K^*:C^B&\longrightarrow& C^A\\
     H&\longmapsto&HK										
\end{array}$$
Ce qui s'\'ecrit brièvement $\lan_K(-)\ladj -.K$.\\
	La notion duale de l'extension \`a gauche est l'extension \`a droite décrite dans la définition suivante:
\end{enumerate}

\begin{defn}
\'Etant donn\'es deux foncteurs $A\stackrel{X}{\longrightarrow}C$ et $A\stackrel{K}{\longrightarrow}B$.
Une extension \`a droite de $X$ le long de $K$ consiste en une paire
  $(R,\epsilon)$ avec $R:B\longrightarrow C$ est un foncteur et $\epsilon:RK\Longrightarrow X$ est une 
	 transformation naturelle universelle. 
	$$	
\xymatrix{
A \ar[rr]^{X} \ar[dr]_{K}& & C\\
    & B \ar@{=>}[u]_{\epsilon}\ar[ur]_R 
} 
$$
	Tout diagramme de la forme
	$$	
\xymatrix{
A \ar[rr]^{X} \ar[dr]_{K}& & C\\
    & B \ar@{=>}[u]_{\gamma}\ar[ur]_{R^{'}}
} 
$$
entraîne l’existence d'une unique transformation naturelle $\alpha:R'\Longrightarrow R$,
($R$ est un objet terminal dans une certaine catégorie), de sorte que $\gamma=\epsilon\circ\alpha K$.\\
	On \'ecrit $\Ran_K(X)=R$
\end{defn}
$\Ran_K(X)=R$ est caractérisé par l'isomorphisme naturel en $H\in C^B$:
\begin{equation}
C^B(H,\Ran_K(X))\cong C^A(HK,X)
\label{eq:2}
\end{equation}

Ainsi, $\Ran_K(X)$ apparaît comme  un objet terminal qui représente le foncteur ensembliste
 $$ \begin{array}{rcl}
C^A(-.K,X):(C^B)^{op}&\ra & \Set\\
     H&\longmapsto&C^A(HK,X)										
\end{array}$$
Si $\Ran_K(X)$ existe pour tout $X\in C^A$ alors $\Ran_K(-)$ est l'adjoint \`a droite de $K^*=-.K$.
Par combinaison avec l'extension \`a gauche, on obtient:
$$\Lan_K(-)\ladj K^*\ladj\Ran_K(-)$$
\begin{exem} La limite est une extension de Kan:\\
Soit un diagramme $D:I\longrightarrow A$. Soient $\1$ la catégorie terminale form\'ee par un seul objet et
$\Delta_1:I\longrightarrow 1$ l'unique foncteur qui relie $I$ \`a $1$.
$$	
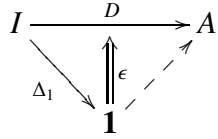
\xymatrix{
I \ar[rr]^{D} \ar[dr]_{\Delta_1}& & A\\
    & \1 \ar@{=>}[u]_{\epsilon}\ar@{-->}[ru]
		}
$$

 On a:\\
 $\ds\lim_I D$ existe si et seulement 
$\Ran_{\Delta_1}(D)$ existe, dans un tel cas, $\ds\lim_I D\cong \Ran_{\Delta_1}(D).$
\end{exem}
\begin{proof}
 Supposons que $\ds\lim_I D$ existe. On pose $l=\ds\lim_I D$. 
$l$ est caractérisé par l'isomorphisme naturel en $a$
$$ A(a,l)\cong A^I(\Delta_a,D)\leqno (*)$$
avec $\Delta_a:I\longrightarrow A$ le foncteur diagonal défini par $\Delta_a(i)=a$.
D'un autre point de vu, on peut voir $a\in A$ comme un foncteur 
$a:\1\longrightarrow A$. La compos\'e $a\circ \Delta_1:I\longrightarrow A$ n'est autre que $\Delta_a$.
(*) s'\'ecrit alors encore
$$ A(a,l)\cong A^I(a\circ\Delta_1,D)$$
Il s'ensuit de (\ref{eq:2}) que $\Ran_{\Delta_1}(D)=l$.
$$	
\xymatrix{
I \ar[rr]^{D} \ar[dr]_{\Delta_1}& & A\\
    & \1 \ar@{=>}[u]_{\epsilon}\ar@/^/[ru]^l\ar@/_/[ru]_a
		}
$$
La réciproque est triviale.
\end{proof}
\begin{exem} L'adjonction est une manifestation de l'extension de Kan:\\
Soient $L$ et $R$ deux foncteurs parallèles de sens opposés: $\oppairi{A}{B}{L}{R}$.\\
On a $L\ladj R$ si et seulement si 
\begin{enumerate}
	\item $\lan_L(1_A)=R$
	\item $\lan_L(L)=LR$
\end{enumerate}
\begin{proof}
Soit $C$ une catégorie. $\oppairi{A}{B}{L}{R}$ induit $\oppairi{C^B}{C^A}{L^*}{R^*}$
avec $L^*(H)=HL$ pour $H\in C^B$ et $R^*(S)=SR$ pour $S\in C^A$. Si $L\ladj R$ alors $R^*\ladj L^*$, ce qui s'\'ecrit
$$C^B(SR,H)\cong C^A(S,HL)\leqno(**)$$
\begin{itemize}
	\item Pour $C=A$ et $S=1_A$, (**) s'\'ecrit:
	$$A^B(R,H)\cong A^A(1_A,HL)$$
	Il s'ensuit que $R\cong \lan_L(1_A)$
	\item Pour $C=B$ et $S=L$, (**) s'\'ecrit:
	$$B^B(LR,H)\cong B^A(L,HL)$$
	Il s'ensuit que $\lan_L(L)=LR$, c'est \`a dire $L\lan_L(1_A)=\lan_(L1_A)$.
	On dit que $L$ préserve $\lan_L(1_A)$.
\end{itemize}
Réciproquement, si $\lan_L(1_A)=R$ et $\lan_L(L)=LR$ alors $L\ladj R$.\\
En effet, $\lan_L(1_A)=R$ fournit une transformation naturelle $\eta:1_A\Longrightarrow RL$.\\
On a $L\eta:L\Longrightarrow LRL$, $\lan_L(L)=LR$ et $1_L:L\Longrightarrow 1_B\circ L$ fournissent une transformation
naturelle unique $\epsilon:LR\Longrightarrow 1_B$ qui vérifie l’identité triangulaire
$\epsilon L\circ L\eta=1_L$
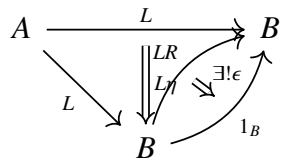
\begin{center}
\begin{tikzcd}[row sep=1cm, column sep=1cm]
A  \ar[dr, "L"', ""{name=L,near end}]
            \ar[rr, "L", ""{name=L, below,bend right}]&& B\\
& B    \ar[ur,bend left, "LR", ""{name=LR, below}]
                \ar[ur, bend right, "1_B"', ""{name=LL}]
%
\arrow[Rightarrow, "\exists !\epsilon", from=LR, to=LL]
\arrow[Rightarrow, from=L,right, "L\eta"]
\end{tikzcd}
\end{center}
\end{proof}
\end{exem}
En considérant le diagramme qui correspond \`a $\lan_L(1_A)=R$
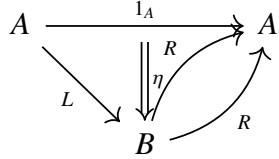
\begin{center}
\begin{tikzcd}[row sep=1cm, column sep=1cm]
A  \ar[dr, "L"', ""{name=L,near end}]
            \ar[rr, "1_A", ""{name=A, below,bend right}]&& A\\
& B    \ar[ur,bend left, "R", ""{name=R, below}]
                \ar[ur, bend right, "R"', ""{name=RR}]
%
\arrow[Rightarrow, from=A,left, "\eta"]
\end{tikzcd}
\end{center}
$\eta$ se factorise suivant $1_R$ et suivant $R\epsilon\circ\eta R$, par universalité de $\eta$,
on obtient $R\epsilon\circ\eta R=1_R$. Il s'ensuit que $L\ladj R$.
\begin{exem} Extension de Kan et Monade:\\
Soit $G:D\longrightarrow C$ un foncteur. Si $G$ admet une extension \`a droite $(T,\epsilon)$  le long de 
$G$ lui même 
$$	
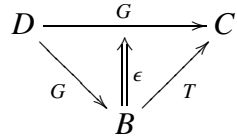
\xymatrix{
D \ar[rr]^{G} \ar[dr]_{G}& & C\\
    & B \ar@{=>}[u]_{\epsilon}\ar[ur]_T
} 
$$

alors il existe deux transformations naturelles $\eta:1_C\Longrightarrow T$ et $\mu:T^2\Longrightarrow T$
de sorte que $(T,\eta,\mu)$ définit une monade sur $C$.
\end{exem}
\begin{proof}
Considérons le diagramme suivant
$$	
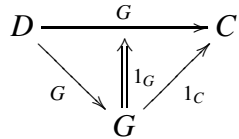
\xymatrix{
D \ar[rr]^{G} \ar[dr]_{G}& & C\\
    & G \ar@{=>}[u]_{1_G}\ar[ur]_{1_C} 
} 
$$
D’après la propriété universelle de $(T,\epsilon)$, il existe une unique transformation naturelle $\eta:1_C\Longrightarrow T$
de sorte que $\epsilon\circ\eta G=1_G$.
La transformation naturelle $\epsilon:TG\Longrightarrow G$ induit une transformation
$\epsilon\circ T\epsilon:T^2G\Longrightarrow G$.
 Le diagramme suivant
$$	
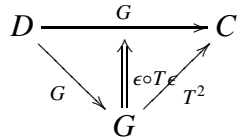
\xymatrix{
D \ar[rr]^{G} \ar[dr]_{G}& & C\\
    & G \ar@{=>}[u]_{\epsilon\circ T\epsilon}\ar[ur]_{T^2} 
} 
$$
avec l’universalité de $(T,\epsilon)$ donne l'existence d'une unique transformation naturelle
$\mu:T^2\Longrightarrow T$ de sorte que $\epsilon\circ\mu G=\epsilon\circ T\epsilon$.

Pour finir, il nous reste \`a voir que $\eta$ et $\mu$ sont respectivement l'unit\'e et la multiplication
de la monade $T$. Le diagramme suivant	
$$
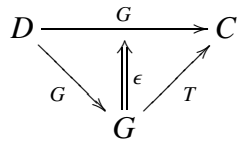
\xymatrix{
D \ar[rr]^{G} \ar[dr]_{G}& & C\\
    & G \ar@{=>}[u]_{\epsilon}\ar[ur]_{T}
} 
$$
assure l'existence d'une unique transformation naturelle $\delta:T\Longrightarrow T$ de sorte que 
$\epsilon\circ\delta G=\epsilon$. Puisque $\epsilon\circ (1_T)G=\epsilon$ alors $\delta=1_T$.
On a aussi:
\begin{eqnarray*}
\epsilon\circ(\mu\circ T\eta)G&=&\epsilon\circ(\mu G\circ T\eta G)\\
                               &=&(\epsilon\circ\mu G)\circ T\eta G\\
															  &=&\epsilon\circ T\epsilon\circ T\eta G\\
																&=&\epsilon\circ T(\epsilon\circ\eta G)\\
																&=&\epsilon\circ T(1_G)\\
																&=&\epsilon
\end{eqnarray*}
Il s'ensuit que $\mu\circ T\eta=1_T$. Un calcul analogue permet de vérifier le reste des identités de l’unité et la multiplication.
\end{proof}
\section{Une formule pour l'extension de Kan}
Soient $A\stackrel{X}{\longrightarrow}C$ et $A\stackrel{K}{\longrightarrow}B$ deux foncteurs de même domaine.
$$	
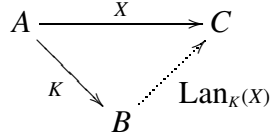
\xymatrix{
A \ar[rr]^{X} \ar[dr]_{K}& & C\\
    & B \ar@{.>}[ur]_{\lan_K(X)}
} 
$$

 Soit $b\in B$.
Si pour $a\in A$, il existe une flèche $K(a)\longrightarrow b$ alors on peut regarder $K(a)$ comme une approximation \`a gauche de $b$ le long de $K$.
Si on veut d\'efinir la valeur de $\lan_K(X)(b)$ alors il faut considérer toutes les approximations possibles de $b$.
Soit $K\downarrow b$ la comma catégorie dont les objets sont les flèches de la forme $K(a)\longrightarrow b$ et  les flèches sont les flèches
$a\longrightarrow a'$ de $A$ qui fournissent des triangles commutatifs dans $B$. Les objets de la catégorie $K{\downarrow} b$ représentent une  paramétrisation des approximations \`a gauche de $b$ le long de $K$. Soit
$$ \begin{array}{rrl}
\pi^b:K\downarrow b\longrightarrow A\\
     (K(a)\longrightarrow b)\longmapsto a										
\end{array}$$
 La composée $X\circ\pi^b$ renvoie une approximation $(K(a)\longrightarrow b)$ vers $X(a)$, il est donc naturel
de penser que $\lan_K(X)(b)$ est la colimite de $X\circ\pi^b$. 
\begin{thm}\label{th1}Condition suffisante d'optimalit\'e:\\
Soient $A\stackrel{X}{\longrightarrow}C$ et $A\stackrel{K}{\longrightarrow}B$ deux foncteurs de même domaine.
\begin{enumerate}
	\item Si pour tout $b\in B$, $\clm(K{\downarrow} b\stackrel{\pi^b}{\rightarrow}A\stackrel{X}{\rightarrow} C)$ existe 
 alors $\lan_K(X)$ existe et
\begin{equation}
\lan_K(X)(b)=\clm(K{\downarrow} b\stackrel{\pi^b}{\rightarrow}A\stackrel{X}{\rightarrow} C)
\label{eq1}
\end{equation}
\item Si pour tout $b\in B$, $\lim(b{\downarrow}K\stackrel{\pi_b}{\rightarrow}A\stackrel{X}{\rightarrow} C)$ existe 
 alors $\Ran_K(X)$ existe et
\begin{equation}
\Ran_K(X)(b)=\lim(b{\downarrow}K\stackrel{\pi_b}{\rightarrow}A\stackrel{X}{\rightarrow} C)
\label{eq2}
\end{equation} 
\end{enumerate}
\end{thm}
Ce théorème a été  démontré dans \cite{ML} et \cite{K}. 
\begin{exem}
On peut \'etendre la  fonction $2^-:\n\longrightarrow\R_+^*$ d\'efinie de façon naturelle \`a un homomorphisme 
de groupes $2^-:\Z\longrightarrow\R_+^*$  en déclarant $2^{-n}=(\frac{1}{2})^n$ qu'on peut l'\'etendre \`a son tour
\`a un homomorphisme $2^-:\Q\longrightarrow\R_+^*$ en déclarant $2^{\frac{1}{n}}=\sqrt[n]{2}$. Il n'existe pas une formule arithmetique
qui permet d'étendre $2^-:\Q\longrightarrow\R_+^*$ \`a une fonction conservant l'ordre $2^-:\R\longrightarrow\R_+^*$. Il s'agit
d'un cas spécial de l'extension d'un foncteur $A\stackrel{X}{\rightarrow}C$ le long d'un foncteur $A\stackrel{K}{\rightarrow}B$.
Regardons $\Q$ et $\R$ comme des catégorie d'ensemble ordonn\'es. Il s'ensuit que $2^-:\Q\longrightarrow\R_+^*$ est un foncteur.
Soit $I:\Q\hookrightarrow\R$ l'injection de $\Q$ dans $\R$. 
$$	
\xymatrix{
\Q \ar[rr]^{2^-} \ar@{^{(}->}[dr]_{I}& & \R_+^*\\
    & \R \ar@{.>}[ur]_{}
} 
$$
Soit $x\in\R$. La catégorie $I\downarrow x$ est $\{q\in\Q:q\leq x\}$. On a
$$\clm(I\downarrow x\stackrel{\pi^x}{\rightarrow}\Q\stackrel{2^-}{\rightarrow} \R_+^*)=\sup\{2^q:q\leq x\}$$
Cette colimite existe car toute partie de $\R$ majorée admet une borne supérieure.
D'après le théorème \ref{th1}, $\lan_I(2^-)$ existe, la formule (\ref{eq1}) s'\'ecrit:
$$\lan_I(2^-)(x)=\sup\{2^q:q\leq x\}.$$
Ainsi, on retrouve la définition usuelle de $2^x$. La formule (\ref{eq2}) donne
$$\Ran_I(2^-)(x)=\inf\{2^q:q\geq x\}.$$
On remarque que dans ce cas les deux extensions \`a gauche et \`a droite coïncident.
\end{exem}
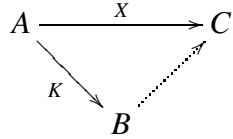
\begin{coro}
Soit 
$$	
\xymatrix{
A \ar[rr]^{X} \ar[dr]_{K}& & C\\
    & B \ar@{.>}[ur] 
} 
$$
Supposons que $A$ est petite et que $B$ est localement petite.
\begin{enumerate}
	\item Si $C$ est cocomplète alors $\lan_K(X)$ existe et est donn\'e par la colimite (\ref{eq1}).
	\item Si $C$ est complète alors $\Ran_K(X)$ existe et est donn\'e par la limite (\ref{eq2}).
\end{enumerate}
\end{coro}
\begin{proof}
Si $A$ est petite alors, pour $b\in B$, la comma catégorie  $K{\downarrow}b$ est petite. Si $C$ est cocomplete alors
la colimite (\ref{eq1}) existe.
\end{proof}
\begin{defn}
Soit
$$	
\xymatrix{
A \ar[rr]^{X} \ar[dr]_{K}& & C\ar[rr]^{G}&&D\\
    & B \ar@{=>}[u]_{\epsilon}\ar[ur]_R 
} 
$$
Avec $\Ran_K(X)=R$, de counit\'e $\epsilon$.
On dit que $G$ préserve $\Ran_K(X)$ lorsque $\Ran_K(GX)\cong GR$, de counit\'e $G\epsilon$:
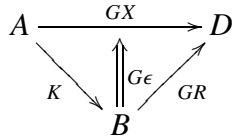
$$
\xymatrix{
A \ar[rr]^{GX} \ar[dr]_{K}& & D\\
    & B \ar@{=>}[u]_{G\epsilon}\ar[ur]_{GR} 
} 
$$
\end{defn}
Les foncteurs qui préservent l'extension \`a gauche sont définies de façon duale.
\begin{prop}
Un adjoint \`a gauche préserve les extensions \`a gauche.
\end{prop}
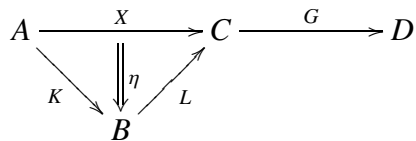
\begin{proof}
Soit
$$	
\xymatrix{
A \ar[rr]^{X} \ar[dr]_{K}& & C\ar[rr]^{G}&&D\\
    & B \ar@{<=}[u]_{\eta}\ar[ur]_L 
} 
$$
Avec $\lan_K(X)=L$, d'unit\'e $\eta$.\\
Si $G$ est l'adjoint \`a gauche d'un foncteur $F:D\longrightarrow C$. Pour $H\in D^B$, on a 
\begin{eqnarray*}
D^B(GL,H)&\cong&C^B(L,FH)\:\:(G\ladj F)\\
          &\cong& C^A(X,FHK)\:\:(\lan_K(X)\cong L)\\
					&\cong& D^A(GX,HK)\:\:(G\ladj F)
\label{eq:}
\end{eqnarray*}
Il s'ensuit que $$\lan_K(GX)\cong GL\cong G\lan_K(X)$$
C'est \`a dire $G$ préserve $\lan_K(X)$
\end{proof}
\begin{defn}
Soit 
$$	
\xymatrix{
A \ar[rr]^{X} \ar[dr]_{K}& & C\ar[rr]^{H^c}&&\Set\\
    & B \ar@{=>}[u]_{\epsilon}\ar[ur]_{\Ran_K(X)} 
} 
$$
On dit que $\Ran_K(X)$ est simple (ou ponctuelle) lorsque les représentables covariants $H^c=C(c,-)$
préservent $\Ran_K(X)$, c'est \`a dire pour $c\in C$, $\Ran_K(H^cX)\cong H^c(\Ran_K(X))$
\end{defn}
\begin{thm}\label{th2}Critère de ponctualité:\\
\begin{enumerate}
	\item $\Ran_K(X)$ est simple $\Leftrightarrow \Ran_K(X)(b)\cong\lim(b{\downarrow}K\stackrel{\Pi}{\rightarrow}A\stackrel{X}{\rightarrow}C)$.
	\item $\lan_K(X)$ est simple $\Leftrightarrow \lan_K(X)(b)\cong\clm(b{\downarrow}K\stackrel{\Pi}{\rightarrow}A\stackrel{X}{\rightarrow}C)$.
\end{enumerate}

\end{thm}
\begin{proof}
$\Rightarrow:$\\
Soit $c\in C$. 
$$
\xymatrix{
A \ar[rr]^{H^cX} \ar[dr]_{K}& & \Set\\
    & B \ar@{<=}[u]_{H^c\epsilon}\ar[ur]_{H^c\Ran_K(X)} 
} 
$$
On a 
\begin{eqnarray*}
C(c,\Ran_K(X)(b))&\cong& \Set^B(H^b,H^c\Ran_K(X))\:\:\:(\mbox{Lemme de Yoneda})\\
                 &\cong& \Set^B(H^b,\Ran_K(H^cX)\:\:(\mbox{$\Ran_K(X)$ est ponctuelle})\\
								&\cong& \Set^A(H^bK,H^cX)\:\:(\mbox{Par définition de $\Ran_K(H^cX)$})\\
								&\cong& Cones(c,X\Pi_b)\\
								&\cong&C(c,\lim_{b{\downarrow}K}X\Pi_b)
\label{eq:}
\end{eqnarray*}
Pour justifier l'isomorphisme $\Set^A(H^bK,H^cX)\cong Cones(c,X\Pi_b)$ il suffit d'\'ecrire
un cône de sommet $c$ sur $X\Pi_b$ et une transformation naturelle $\alpha:H^bK\Longrightarrow H^cX$
et remarquer la correspondance bijective entre les deux.\\
$\Leftarrow:$\\
Si $\Ran_K(X)(b)\cong \lim_{b{\downarrow}K}X\Pi_b)$ alors
 $$H^c(\Ran_K(X)(b))\cong H^c(\lim_{b{\downarrow}K}X\Pi_b)\cong\lim_{b{\downarrow}K}H^cX\Pi_b\cong \Ran_K(H^cX)(b)$$
car les représentables covariants préservent les limites. 
\end{proof}
\begin{exem}
Soit $A\stackrel{X}{\rightarrow}C$ un foncteur. En appliquant la propriété universelle qui
définit l'extension \`a droite de $X$ le long de $1_A$ on trouve $\Ran_{1_A}(X)\cong X$.
Pour la même raison $\Ran_{1_A}(H^cX)\cong H^cX\cong H^c(\Ran_{1_A}(X))$.
Il s'ensuit que $\Ran_{1_A}(X)$ est ponctuelle. D'après le théorème \ref{th2}, 
$$X(a)\cong \lim(a{\downarrow}A\stackrel{\Pi}{\rightarrow}A\stackrel{X}{\rightarrow}C)$$
En particulier, si $C=\Set$ alors on obtient
$$X(a)\cong\lim_{a{\downarrow}A}X\Pi\cong Cones(1,X\Pi)\cong \Set^A(H^a,X)$$
Ainsi, on retrouve le  lemme de Yoneda. Par dualit\'e, on obtient:
$$X(a)\cong \clm(A{\downarrow}a\stackrel{\Pi}{\rightarrow}A\stackrel{X}{\rightarrow}C),$$
c'est ce qu'on appelle le lemme de coYoneda.
\end{exem}
\begin{coro}
Soit 
$$
\xymatrix{
A \ar[rr]^{X} \ar[dr]_{K}& & C\\
    & B  
} 
$$
Supposons que $X$ admet ponctuellement une extension \`a droite. Si $K$ est pleinement fidel
alors $\Ran_K(X)K\cong X$
\end{coro}
\begin{proof}
Pour $a\in A$,
\begin{eqnarray*}
\Ran_K(X)(K(a))&\cong&\lim_{K(a){\downarrow}K}[X\Pi]\\
                &\cong& \lim_{a/A}[X\Pi]\\
								&\cong&X\Pi(1_a)\mbox{  ($1_a$ est initial dans $a/A$)}\\
								&\cong& X(a)
\label{eq:}
\end{eqnarray*}
\end{proof}
\begin{thm}\label{th3}
Pour toute petite catégorie $C$, l'extension \`a gauche du plongement de Yoneda $y:C\hookrightarrow \Set^{C^{\op}}$ le long de $y$ lui même
est le foncteur identit\'e.
\end{thm}
\begin{proof}
Soit $C$ une catégorie petite.
$$
\xymatrix{
C \ar@{^{(}->}[rr]^{y} \ar@{^{(}->}[dr]_{y}& & \Set^{C^{\op}}\\
    & \Set^{C^{\op}} \ar@{<=}[u]_{1_y}\ar@{-->}[ur]_{1\cong\lan_y(y)} 
} 
$$
$\Set^{C^{\op}}$ est une catégorie bicomplète.  Pour $F:C^{\op}\longrightarrow\Set$, on a
$$\begin{array}{rcll}
\lan_y(y)(F)&\cong&\clm(y{\downarrow}F\stackrel{\Pi}{\rightarrow}C\stackrel{y}{\hookrightarrow} \Set^{C^{\op}})&\mbox{(Théorème \ref{th1})}\\
             &\cong&\clm(\int F\stackrel{\Pi}{\rightarrow}C\stackrel{y}{\hookrightarrow} \Set^{C^{\op}})&(y{\downarrow}F\cong\int F)\\
						&\cong&\underset{(c,x)\in\int F}\clm H_c&\\
						&\cong& F &\mbox{(Théorème de densit\'e)}
\end{array}$$
\end{proof}
\begin{exer}
Soit $F:C\longrightarrow E$ un foncteur. On suppose que 
$C$ est petite et que $E$ est localement petite et est co-complète.
\begin{enumerate}
	\item Notons $y:C\hookrightarrow \Set^{C^{\op}}$ le plongement de Yoneda. Montrer que $F$ admet une extension
	\`a gauche le long de $y$.
	\item\label{def} Si $\lan_y(F)$ admet un adjoint \`a droite $R:E\longrightarrow \Set^{C^{\op}}$ alors comment doit s'exprimer
	$R_e(c)$ pour $e\in E$ et $c\in C^{\op}$?
	\item Prouver que $R$ trouvé dans (\ref{def}) est effectivement l'adjoint \`a droite de $\lan_y(F)$.
\end{enumerate}
\end{exer}
{\bf Solution:}\\
$$
\xymatrix{
C \ar[rr]^{F} \ar@{^{(}->}[dr]_{y}& & E\\
    & \Set^{C^{\op}} \ar@{-->}[ur]_{\lan_y(F)} 
} 
$$
\begin{enumerate}
	\item Puisque $C$ est petite et $E$ est co-complete alors les colimites 
	$\clm(y{\downarrow}G\stackrel{\Pi}{\rightarrow}C\stackrel{F}{\rightarrow} E)$ existent et définissent, d'après le théorème \ref{th1},
	$\lan_y(F)(G)$.
	\item Si $\lan_y(F)\ladj R$ alors pour $e\in E$ et $c\in C^{\op}$
	$$\begin{array}{rcll}
R_e(c)&\cong&\Set^{C^{\op}}(y_c,R_e)&\mbox{(Lemme de Yoneda)}\\
             &\cong&E(\lan_y(F)(y_c),e)&(\lan_y(F)\ladj R)\\
						&\cong&E(F(c),e)&\mbox{($y$ est pleinement fidel)}
	\end{array}$$
	\item Soit $X\in \Set^{C^{\op}}$. Pour $e\in E$, on a
	$$\begin{array}{rcll}
\Set^{C^{\op}}(X,R_e)&\cong&\Set^{C^{\op}}(\underset{(c,x)\in\int X}\clm y_c,R_e)&\mbox{(Théorème de densit\'e)}\\
             &\cong&\ds\lim_{(c,x)\in\int X}\Set^{C^{\op}}(y_c,R_e)&\\
						&\cong&\ds\lim_{(c,x)\in\int X}R_e(c) &\mbox{(Lemme de Yoneda)}\\
						&\cong&\ds\lim_{(c,x)\in\int X}E(F(c),e)&\mbox{(Formule de $R$)}\\
						&\cong&E(\underset{(c,x)\in\int X}\clm F(c),e)&\\
						&\cong&E(\underset{y{\downarrow}X}\clm F\Pi,e)&\\
						&\cong&E(\lan_y(F)(X),e)&\\
	\end{array}$$
						
\end{enumerate}


\begin{thebibliography}{}
\bibitem{K}  Kelly, G. M.: Basic concepts of enriched category theory.  Repr. Theory Appl. Categ 10, 1-136 (2005)
\bibitem{HR} Hinze R.:  Kan Extensions for Program Optimisation Or: Art and Dan Explain an Old Trick. In: Gibbons J., Nogueira P. (eds) Mathematics of Program Construction. MPC 2012. Lecture Notes in Computer Science, vol 7342. Springer, Berlin.
\bibitem{ML}  Mac Lane, S.: Categories for the working mathematician, 2nd ed. Graduate Texts in Mathematics, vol. 5. Springer, New York (1998)
 
\bibitem{RE} Emily Riehl, Category Theory in Context. Courier Dover Publications 2017.
 
 
\end{thebibliography}
\end{document}